\documentclass[draft]{amsart}
\usepackage{amssymb} % SRC Specials: DVI [Inverse] Search
\newtheorem{theorem}{Theorem}[section]
\newtheorem{lemma}[theorem]{Lemma}
\newtheorem{proposition}[theorem]{Proposition}
\newtheorem{corollary}[theorem]{Corollary}

\theoremstyle{definition}
\newtheorem{definition}[theorem]{Definition}

\newtheorem{example}[theorem]{Example}

\newtheorem{remark}[theorem]{Remark}
\theoremstyle{approach}

\numberwithin{equation}{section}

\begin{document}
\setcounter{page}{1}
%%%%%%%%%%%%%%%%%%%%%%%%%%%%%%%%%%%%%%%%%%%%%%%%%%%%%%%%%%%%
\title[amenability modulo an ideal  of FR$\acute{E}$CHET algebras]{amenability modulo an ideal  of FR$\acute{E}$CHET algebras}
%%%%%%%%%%%%%%%%%%%%%%%%%%%%%%%%%%%%%%
\author[S. Rahnama and A. Rejali]{S. Rahnama and A. Rejali}
%%%%%%%%%%%%%%%%%%%%%%%%
%%%%%%%%%%%%%%%%%%%%%%%%

%%%%%%%%%%%%%%%%%%%%%%%%%%%%
% Subject class; see http://www.ams.org/mathscinet/msc/msc2010.html

%%%%%%%%%%%%%%%%%%%%%%%%%%%%,

\begin{abstract}
Amenability modulo an ideal of a Banach algebra 
have been defined and studied. 
In this paper we introduce the concept of amenability modulo an
ideal of a  Fr$\acute{e}$chet algebra and investigate some known 
results about amenability modulo an ideal of a Fr$\acute{e}$chet algebra.
Also we show that a  Fr$\acute{e}$chet algebra $(\mathcal A,p_n)_{n\in\mathbb{N}}$ 
is amenable modulo an ideal  if and only if $\mathcal A$ is isomorphic
to a reduced inverse limit of amenable modulo an ideal of Banach algebras. 
\end{abstract}

\maketitle \setcounter{section}{-1}

\section{\bf Introduction}
Some parts of theory of Banach algebras, have been introduced and 
studied for Fr$\acute{e}$chet algebras. For example 
the notion of amenability of a Fr$\acute{e}$chet algebra and their 
applications to harmonic analysis  was
introduced by Helemskii in \cite{Hel4} and \cite{Hel5}, and was
 covered and studied by  Pirkovskii \cite{Pir}. Also in \cite{Law},
approximate amenability and approximate contractibility of 
Fr$\acute{e}$chet algebras was introduced and investigated.
Furthermore, in \cite{ARR1} we and Abtahi studied weak amenability of
Fr$\acute{e}$chet algebras and generalized some  results
related to weak amenability of Banach algebras for
Fr$\acute{e}$chet algebras.
In \cite{Rahimi Amini} Amini and Rahimi introduced the notion of amenability modulo an
 ideal of Banach algebras. They showed that  amenability of the semigroup algebra $\ell^{1}(S)$
modulo ideals by certain classes of group congruence of $S$ is 
equivalent to the amenability of $S$.
Then Rahimi and Tahmasebi \cite{Rahimi1} continued this vertification
 and studied basic properties of amenability modulo an ideal such as
 virtual diagonal modulo an ideal, approximate diagonal 
modulo an ideal and contractibility modulo an ideal for Banach algebras.

In the present work, we continue our study on amenability  of
Fr$\acute{e}$chet algebras.
We generalize
some basic definitions and   results about the concept  of amenability modulo an ideal
in Banach algebra case for Fr$\acute{e}$chet algebras. According to the definition of
 amenability  modulo an ideal  of Banach
algebras,  we introduce the concept of 
amenability modulo an ideal  of
Fr$\acute{e}$chet algebras. Then we verified some concept in the Banach algebra case, for Fr$\acute{e}$chet
algebras.
The remainder of the paper is organized as follows.
Section $1$ presents some preliminaries and basic results and definitions about locally convex spaces and 
Fr$\acute{e}$chet algebras. In section $2$ we define the notion of amenability modulo
an ideal for a Fr$\acute{e}$chet algebra.  As the main result of this section we investigate the 
relation of the amenability modulo an ideal of a Fr$\acute{e}$chet algebra
with amenability modulo an ideal of Banach algebras which form this Fr$\acute{e}$chet algebra.
More precisely  we show that if 
 $\mathcal A=\underleftarrow{\lim}\mathcal A_n$
be an  Arens-Michael decomposition of $\mathcal A$ and $I=\underleftarrow{\lim}\overline{I_n}$
 be an Arens-Michael decomposition of $I$, 
 then $\mathcal A$ is amenable modulo $I$
 if and only if for each $n\in\mathbb{N}$, $\mathcal A_n$
 is amenable modulo $\overline{I_n}$. Also we discuss about the relation between amenability 
 modulo an ideal of a  Fr$\acute{e}$chet algebra and the amenability of the quotient algebra 
of this  Fr$\acute{e}$chet algebra. Moreover we provide some examples of amenable modulo an 
ideal of  Fr$\acute{e}$chet algebras which are not amenable.
Section $3$  describes the notions of locally bounded approximate identity
  modulo an ideal for a   Fr$\acute{e}$chet
algebra. In this section we prove that all amenable modulo an ideal 
  Fr$\acute{e}$chet algebras have locally bounded approximate identities modulo an ideal.

\section{\bf Preliminaries}

In this section, we first exhibit basic  definitions and results 
related to locally convex spaces and also Fr$\acute{e}$chet
algebras, which will be used throughout the paper. We refer the reader to 
\cite{Gold}, \cite{Hel2}, \cite{Hel3}, \cite{MA} and \cite{ME} for these results.

By a  locally convex space $E$ we mean a topological
vector space in which the origin has a local base of absolutely
convex absorbing sets. 
 We denote by $(E,p_{\alpha})$, a
locally convex space $E$ with a fundamental system of seminorms
$(p_{\alpha})_{\alpha}$.

If $(E,p_{\alpha})_{\alpha\in A}$ and $(F,q_{\beta})_{\beta\in
B}$ be locally convex spaces, by applying  \cite[Proposition 22.6]{ME}
 a linear mapping $T: E\to F$ is continuous if and
only if for each $\beta\in B$ there exist an $\alpha\in A$ and
$C>0$, such that
$$q_{\beta}(T(x))\leq Cp_{\alpha}(x),$$ for all $x\in E$.
Also by  \cite[page 24]{Hel2}, for locally convex spaces  $(E,p_{\mu})$, $(F,q_{\lambda})$ and
$(G,r_{\nu})$, the  bilinear map $\theta:E\times F\to
G$  is jointly continuous if and
only if for any ${\nu_0}$ there exist ${\mu_0}$ and ${\lambda_0}$
such that the bilinear map
$$
\theta:(E,p_{\mu_0})\times (F,q_{\lambda_0})\longrightarrow
(G,r_{\nu_0})
$$
is jointly continuous.  Separate continuity of a bilinear map also defined in \cite{Shi}.
In fact the bilinear map
$f:E\times F\rightarrow G$ is said to be separately continuous
if all partial maps $f_x:F\rightarrow G$ and $f_y:E\rightarrow G$
defined by $y\mapsto f(x,y)$ and $x\mapsto f(x,y)$, respectively,
are continuous for each $x\in E$ and $y\in F$. 
  By applying \cite[chapter.III.5.1]{Shi}, we have the
 fact that separate continuity implies joint continuity for 
  Fr$\acute{e}$chet spaces and in particular, Banach 
spaces.

By a  topological algebra we mean a linear 
associative algebra   ${\mathcal A}$,
whose underlying vector space is a topological 
vector space such that  the multiplication
$${\mathcal A}\times{\mathcal A}\longrightarrow {\mathcal A}, \;\;\;\; (a,b)\mapsto ab$$
is a separately continuous mapping; see \cite{Pir}. An outstanding 
particular class of topological algebras is the class of Fr$\acute{e}$chet algebras. A
Fr$\acute{e}$chet algebra, denoted by $(\mathcal A, p_n)$, is a
complete topological algebra, whose topology is given by the
countable family of increasing submultiplicative seminorms; see
\cite{Gold} and \cite{Hel3}. Also every closed subalgebra of a
Fr$\acute{e}$chet algebra is clearly a Fr$\acute{e}$chet algebra.

For  a  Fr$\acute{e}$chet algebra $(\mathcal A,p_n)$,  a
locally convex $\mathcal A$-bimodule is a locally convex
topological vector space $X$  with an
$\mathcal A$-bimodule structure  such that the corresponding mappings are
separately continuous.
 Let $(\mathcal A,p_n)$ be a Fr$\acute{e}$chet algebra and $X$ be a 
locally convex  $\mathcal A$-bimodule.  Following \cite{Hel3}, a continuous  
derivation of $\mathcal A$ into $X$ is a continuous mapping 
$D$ from $\mathcal A$ into $X$ such that 
$$D(ab)=a.D(b)+D(a).b,$$
for all $a,b\in \mathcal A$. Furthermore for each 
$x\in X$ the mapping $\delta_x:\mathcal A\to X$ defined by
$$\delta_x(a)=a.x-x.a\;\;\;\;\;(a\in\mathcal A),$$
is a continuous derivation and is called the inner derivation associated with 
$x$.

We recall definition of an inverse limit from \cite{F}. Let 
 $(E_{\alpha})_{{\alpha}\in \Lambda}$ be a family of algebras, 
where  $\Lambda$ is  a directed  set. Also suppose that $f_{\alpha\beta}$ is a
 family 
 of homomorphisms defined from  $E_{\beta}$ into $E_{\alpha}$
for any ${\alpha,\beta\in \Lambda}$, with $\alpha\leq \beta$. A family 
$\{(E_{\alpha}, f_{\alpha\beta})\}$ is called a  projective system of algebras, if it 
has the above relation and in addition  
 satisfies  the following condition 
$$f_{\alpha\gamma}=f_{\alpha\beta}\circ f_{\beta\gamma},\;\;\; (\alpha,\beta,\gamma\in\Lambda,\alpha\leq\beta\leq\gamma). $$
Now consider the cartesian product algebra  $F=\prod_{\alpha\in\Lambda}E_{\alpha}$, and a  subset of $F$,
$$E=\{x=(x_{\alpha})\in F:x_{\alpha}=f_{\alpha\beta}(x_{\beta}),\;\; \alpha\leq \beta\}. $$
Then $E$ is the projective (or inverse) limit of the  projective system
 $\{(E_{\alpha}, f_{\alpha\beta})\}$ and we denote it
  by $E= \underleftarrow{\lim}\{(E_{\alpha}, f_{\alpha\beta})\}$
or simply $E=\underleftarrow{\lim}E_{\alpha}$.

Now let $(\mathcal A, p_{\lambda})$ be a locally convex algebra.
Obviously for each ${\lambda\in \Lambda}$, $N_{\lambda}= \ker p_{\lambda}$  
 is an ideal  in $\mathcal A$ and
 $\frac{\mathcal A}{N_{\lambda}}$ is a   normed algebra. 
Suppose that $\varphi_{\lambda}:\mathcal A\to\mathcal A_{\lambda},\varphi_{\lambda}(x)=x_{\lambda}=x+{N_{\lambda}} $
be the corresponding quotient map. It  is clear that $\varphi_{\lambda}$ is  a continuous surjective 
homomorphism. Now if  $\lambda, \gamma\in\Lambda$ with $\lambda\leq \gamma$, one 
has $N_{\lambda}\subseteq N_{\gamma}$. So that the linking maps
$$\varphi_{\lambda\gamma}:\frac{\mathcal A}{N_{\gamma}}\to\frac{\mathcal A}{N_{\lambda}},\;\;\;\varphi_{\lambda\gamma}(x+N_{\gamma})=x+{N_{\lambda}}, $$
are well defined continuous surjective homomorphisms such that
$\varphi_{\lambda\gamma}\circ \varphi_{\gamma}=\varphi_{\lambda}$. Hence $\varphi_{\lambda\gamma}$'s
have unique extentions to continuous homomorphisms between
the Banach algebras $\mathcal A_{\gamma}$ and $\mathcal A_{\lambda}$, we use the symbol 
$\varphi_{\lambda\gamma}$ for the extentions too, which  
$\mathcal A_{\gamma}$ is the compeletion of $\frac{\mathcal A}{N_{\gamma}}$.
The families 
$(\frac{\mathcal A}{N_{\lambda}},\varphi_{\lambda\gamma} )$ [respectively, $(\mathcal A_{\lambda},\varphi_{\lambda\gamma} )$],
form inverse system of 
normed,[ respectively, Banach] algebras.  We denote the  corresponding inverse limits 
by $\underleftarrow{\lim}\frac{\mathcal A}{N_{\lambda}}$ and $\underleftarrow{\lim}\mathcal A_{\lambda}$.
Moreover if the initial algebra $\mathcal A$ is complete, one has 
$$\mathcal A=\underleftarrow{\lim}\frac{\mathcal A}{N_{\lambda}}  =\underleftarrow{\lim}\mathcal A_{\lambda},$$
up to topological isomorphisms.
Now let $\mathcal A$ be a Fr$\acute{e}$chet algebra 
with fundamental system of increasing
 submultiplicative seminorms $(p_n)_{n\in\mathbb{N}}$.
For each $n\in\mathbb{N}$ let $\varphi_n:\mathcal A\to \frac{\mathcal A}{\ker p_n}$
be the quotient map. Then $\frac{\mathcal A}{\ker p_n}$ is naturally a normed
algebra, normed by setting $\|\varphi_n(a)\|_n=p_n(a)$ for each $a\in\mathcal A$
and the  compeletion  $(\mathcal A_n, \|.\|_n)$ 
is a Banach algebra. We call the map $\varphi_n$  from $\mathcal A$ into
$\mathcal A_n$, the canonical map. It is important to note that $\varphi_n(\mathcal A)$ is a dense 
subalgebra of $\mathcal A_n$ and  in general $\mathcal A_n\neq\varphi_n(\mathcal A)$.

The above is the Arense-Michael decomposition of $\mathcal A$, 
which expresses Fr$\acute{e}$chet algebra as an reduced inverse limit of Banach algebras. 
Now choose an Arens-Michael decomposition  $\mathcal A=\underleftarrow{\lim}\mathcal A_n$
and let $I$ be a closed ideal of $\mathcal A$.
Then it is easy to see that $I=\underleftarrow{\lim}\overline{I_n}$ is an 
Arens-Michael decomposition of $I$, where $\varphi_{n}:I\to I_n$
is the canonical map.(see \cite{Pir}).

According  to \cite{Pir} if $\mathcal A$ is a locally convex algebra
and $X$ is a left locally convex  $\mathcal A$-module, then  a continuous
seminorm $q$ on $X$  is m-compatible if there exists a continuous 
submultiplicative seminorm $p$ on  $\mathcal A$ such that 
$$q(a.x)\leq p(a)q(x),\;\;\;(a\in \mathcal A, x\in X).$$
Also by \cite[3.4]{Pir1} if $\mathcal A$ is a Fr$\acute{e}$chet algebra
and $X$ is a complete left  $\mathcal A$-module
with a jointly continuous left module action, then the topology
on $X$ can be determined by a directed family of m-compatible
seminorms.

\section{\bf Amenability  modulo an ideal of a  Fr$\acute{e}$chet algebra }
Let $\mathcal A$ be a Banach algebra and $I$ be a closed ideal of $\mathcal A$. According to
\cite{Rahimi  Amini}, $\mathcal A$ is amenable modulo $I$, if for every Banach 
$\mathcal A$-bimodule $E$ such that $I.E=E.I=0$ and every derivation $D$ from $\mathcal A$
into $E^*$, there exists $\varphi\in E^*$ such that
$$D(a)=a.\varphi-\varphi.a,  \;\;\;\;(a\in\mathcal A\setminus I).$$
We commence with the  definition of amenability  modulo an ideal
for a Fr$\acute{e}$chet algebra. We extend some results of \cite{Rahimi1}, for Fr$\acute{e}$chet algebras.
Recall that for the algebra $\mathcal A$,
$$\mathcal A.\mathcal A=\{a.b:\;\;\; a,b\in\mathcal A\}. $$
Also $\mathcal A^2$ is  the linear span of $\mathcal A.\mathcal A$.
\begin{definition}\label{D1}\rm
Let $(\mathcal A, p_{n})$ be a Fr$\acute{e}$chet algebra and $I$ be a closed ideal
of  $\mathcal A$. We
call $\mathcal A$ amenable modulo $I$,  if  for every Banach  
$\mathcal A$-bimodule $E$ such that $E.I=I.E=0$ each
continuous derivation from $\mathcal A$ into $E^*$ is
 inner on $\mathcal A\setminus I$.
\end{definition}
Note that the concept of amenability modulo an ideal
for a Fr$\acute{e}$chet algebra $\mathcal A$ coincides with
the concept of amenability modulo an ideal, in the case where 
$\mathcal A$ is a Banach algebra. Also in view of \cite[Theorem 9.6]{Pir}
each amenable Fr$\acute{e}$chet algebra is   amenable modulo $I$ for each closed 
ideal $I$. But at the end of this section we show that in general the converese of it,
is not true. An easy computation shows that   
  if a Fr$\acute{e}$chet algebra $\mathcal A$ is amenable modulo  $I=\{0\}$,
then $\mathcal A$ is amenable. Henceforth amenability modulo an ideal 
for a Fr$\acute{e}$chet algebra is a generalization of the concept of the
amenability  for a Fr$\acute{e}$chet algebra.

As for first result we extend  \cite[Theorem 8]{Rahimi1} for 
Fr$\acute{e}$chet algebras.  
The proof is similar to the Banach algebra case.
\begin{proposition}\label{P1}
Let $\mathcal A$ be a Fr$\acute{e}$chet algebra and
$I$ be a closed ideal of $\mathcal A$ and $\mathcal A$ be
 amenable modulo $I$. Suppose that
$\mathcal{B}$ is a  Fr$\acute{e}$chet algebra and
$J$ is a closed ideal of $\mathcal B$. Let $\varphi: \mathcal A\to \mathcal B$
be a continuous homomorphism with dense range
such that $\varphi(I)\subseteq J$. Then $\mathcal B$ is amenable modulo $J$.
\end{proposition}
\begin{proof}
 Suppose that $E$ is a Banach $\mathcal B$-bimodule 
such that $J.E=E.J=0$ and $D:\mathcal B\to E^*$ is a continuous derivation.
Then $E$ becomes a Banach $\mathcal A$-bimodule
with the module actions defined
by 
$$a.x=\varphi(a)x \;\;\;\;\;  and\;\;\;\;\;x.a=x\varphi(a)\;\;\;\;(a\in \mathcal A,x\in E).$$
Obviously  $I.E=E.I=0$ and also $D\circ \varphi:\mathcal A\to E^*$ 
is a continuous derivation. On the other hand $\mathcal A$ is amenable modulo $I$,
 so there exists $\eta\in E^*$ such that $(D\circ \varphi) (a)= a.\eta-\eta.a$ 
on $\mathcal A\setminus I$. Now if $b\in\mathcal B\setminus J$, then
there exists a net $(a_{\alpha})_{\alpha}\subseteq\mathcal A$
such that $b=\lim_{\alpha}\varphi(a_{\alpha})$.
Since $\varphi(I)\subseteq J$, we may assume that
$(a_{\alpha})_{\alpha}\subseteq\mathcal A\setminus I$. Now  we have 
\begin{eqnarray*}
D(b)&=&\lim_{\alpha}(D\circ\varphi)(a_{\alpha})\\
&=&\lim_{\alpha}(a_{\alpha}.\eta-\eta.a_{\alpha})\\
&=&\lim_{\alpha}(\varphi(a_{\alpha})\eta-\eta\varphi(a_{\alpha}))\\
&=& b\eta-\eta b.
\end{eqnarray*}
Thus $\mathcal B$ is amenable modulo $J$.
\end{proof}
In \cite [Theorem 9.5]{Pir}, Pirkovskii asserts that a Fr$\acute{e}$chet algebra 
 $\mathcal A$ is amenable if and only if   $\mathcal A$ is isomorphic
to a reduced inverse limit of amenable   Banach algebras.
In the following theorem, we extend  this result as the main result of this section,
 for amenability modulo an ideal 
of Fr$\acute{e}$chet algebras.

During this section, for simplicity's sake we use the notation $I_n=\varphi_n(I)$, where
$\varphi_n:\mathcal A\to \mathcal A_n$ is the canonical map for each $n\in\mathbb{N}$.
\begin{theorem}\label{T1}
Let $(\mathcal A,p_n)$ be a Fr$\acute{e}$chet algebra and
$I$ be a closed ideal of $\mathcal A$. Then the following assertions are
equivalent;
\begin{enumerate}
\item[(i)] $\mathcal A$ is amenable modulo $I$.
 \item[(ii)]
For each Arens-Michael decomposition of
$\mathcal A=\underleftarrow{\lim}\mathcal A_n$
all $\mathcal A_n$'s are amenable Banach algebras modulo $\overline{I_n}$'s, 
where $I=\underleftarrow{\lim}\overline{I_n}$ is an 
Arens-Michael decomposition of $I$. 
\end{enumerate}
\end{theorem}
\begin{proof}
$(i)\Rightarrow (ii)$. Let $\mathcal A=\underleftarrow{\lim}\mathcal A_n$
be an  Arens-Michael decomposition of $\mathcal A$ and
 $\mathcal A$ be amenable  modulo ${I}$. Since $\varphi_n:\mathcal A\to \mathcal A_n$ is 
a continuous homomorphism with dense range, by Proposition \ref{P1},
 $\mathcal A_n$ is amenable  modulo $\overline{I_n}$
for each $n\in\mathbb{N}$.\\
$(ii)\Rightarrow (i)$. 
Let $\mathcal A_n$ be amenable  modulo $\overline{I_n}$
for each $n\in\mathbb{N}$ and let $E$ be  a Banach $\mathcal A$-bimodule
such that  $E.I=I.E=0$ and $D:\mathcal A\to  E^*$ be a continuous derivation. 
Since $(E^*, \| \cdot\|) $ is a Banach space, continuity of $D$ implies the 
existence of a constant $C>0$ and an $n_0\in\mathbb{N}$
such that $$\| D(a)\|\leq Cp_{n_0}(a),\;\;\;\;\;\;( a\in\mathcal A).$$
Since $\{p_n\}_{n}$ is a fundamental system of increasing 
seminorms, we have $\| D(a)\|\leq C  p_m(a)$,  for each $m\geq n_0$.
Thus $\ker p_m\subseteq \ker D$, for each $m\geq n_0$, so the function 
 $$D_m:\frac{\mathcal A}{\ker p_m}\to  E^* , \;\;\;\;D_m(a+\ker p_m)=D_m(a_m)=D(a),\;\;\;(a\in\mathcal A),$$
is well defined and is a continuous derivation on $\frac{\mathcal A}{\ker p_m}$,
for each $m\geq n_0$. The unique extention of $D_m$ to the Banach algebra 
$\mathcal A_m$ is a continuous derivation is denoted by $D_m$.
On the other hand since $(E, \|\cdot \|)$ is a Banach $\mathcal A$-bimodule,
the norm on $E$ is $m$-compatible, so there exists $n_1\in\mathbb{N}$
such that $\|a.x\|\leq p_{n_1}(a)\|x\|$, for all $ a\in \mathcal A$
 and $x\in E$. (see \cite[Page 7]{Pir}). Thus 
 $$\|a.x\|\leq p_{m}(a)\|x\|, \;\;\;\;\;\;   (a\in \mathcal A,\;x\in E,\;m\geq n_1).$$
Then $E$ is a Banach $\mathcal A_m$-bimodule in a natural way for 
each $m\geq n_1$.(see \cite[Page 7]{Pir}). In conclusion, since $E.I=I.E=0$
we have the following
$$E.I_m=I_m.E=0, \;\; (m\geq n_1),$$ and so 
$$E.\overline{I_m}=\overline{I_m}.E=0,  \;\;\;( m\geq n_1).$$
Now set $n=\max\{n_1, n_0\}$. In view of the above arguments
$$D_n:\frac{\mathcal A}{\ker p_n}\to  E^* , \;\;\;\;D_n(a+\ker p_n)=D(a),$$
is a continuous derivation and $\overline{I_n}.E=E.\overline{I_n}=0$ and
the unique extention of $D_n$ to the Banach algebra 
$\mathcal A_n$ is a continuous derivation also denoted by $D_n$.
Since $\mathcal A_n$ is amenable modulo $\overline{I_n}$, there exists
$\varphi\in E^*$ such that
$$D_n(a_n)=a_n.\varphi-\varphi.a_n \;\;\;\;( a_n\in \mathcal A_n\setminus \overline{I_n}). $$
Therefore $D(a)=a.\varphi- \varphi.a$ for each $a\in \mathcal A\setminus I$. Note that
if $a=(a_n)\notin I=\underleftarrow{\lim}\overline{I_n}$, then there is an $n_0\in\mathbb{N}$
such that $a_{n_0}\notin \overline{I_{n_0}}$. Since the mappings $\varphi_{m{n_0}}$ 
are defined by $\varphi_{m{n_0}}(a_m)=a_{n_0}$, then $a_m\notin \overline{I_m}$, for
each $m\geq n_0$. So we can assume that $a_n\in \mathcal A_n\setminus \overline{I_n}$ for each $n\geq n_0$.
 So $\mathcal A$ is amenable modulo $I$.
\end{proof}
The following theorem is a generalization of \cite[Theorem 1]{Rahimi  Amini}. The 
proof is completely different from the Banach algebra case.
\begin{theorem}\label{t12}
Let $(\mathcal A,p_n)$ be a Fr$\acute{e}$chet algebra and
$I$ be a closed ideal of $\mathcal A$. Then the following statements hold;
\begin{enumerate}
\item[(i)]  If $\mathcal A$ is amenable modulo $I$, then $\frac{\mathcal A}{I}$
is amenable.
 \item[(ii)]
If $\mathcal A$ is amenable modulo $I$ and $I$ is amenable, then 
$\mathcal A$ is amenable.
 \item[(iii)] If $\frac{\mathcal A}{I}$ is amenable and $I^2=I$, then 
$\mathcal A$ is amenable modulo $I$.
\end{enumerate}
\end{theorem}
\begin{proof}
$(i)$. Let $\mathcal A=\underleftarrow{\lim}\mathcal A_n$ be an
 Arens-Michael decomposition of $\mathcal A$ and   $I=\underleftarrow{\lim}\overline{{I_n}}$ be 
an Arens-Michael decomposition of $I$.  By the hypothesis $\mathcal A$ is amenable modulo $I$,
so  by using Theorem \ref{T1}, 
$\mathcal A_n$ is amenable modulo $\overline{I_n}$,  for each $n\in\mathbb{N}$. 
Moreover 
$\frac{\mathcal A}{I}=\underleftarrow{\lim}\frac{\mathcal A_n}{\overline{I_n}}$
is an Arens-Michael decomposition of $\frac{\mathcal A}{I}$, by \cite[Theorem 3.14]{F}. The 
assertion now follows from \cite[Theorem 9.5]{Pir} and \cite[Theorem 1]{Rahimi  Amini}.\\

$(ii)$. Let $\mathcal A=\underleftarrow{\lim}\mathcal A_n$ be an
 Arens-Michael decomposition of $\mathcal A$. By \cite[Theorem 9.5]{Pir}
it is sufficient to show that $\mathcal A_n$ is amenable for each
$n\in\mathbb{N}$, for amenability of $\mathcal A$. Suppose that 
 $I=\underleftarrow{\lim}\overline{I_n}$ is 
an Arens-Michael decomposition of $I$. Since $\mathcal A$ 
is amenable modulo $I$, Theorem \ref{T1}, implies that 
$\mathcal A_n$ is amenable modulo $\overline{I_n}$, for each $n\in\mathbb{N}$
and by \cite[Theorem 9.5]{Pir}, $\overline{I_n}$, is amenable
 for each $n\in\mathbb{N}$.
Consequently  by \cite[Theorem 1]{Rahimi Amini}, $\mathcal A_n$
is amenable for each $n\in\mathbb{N}$.\\
$(iii)$. Let $E$ be a Banach $\mathcal A$-bimodule 
 such that $I.E=E.I=0$ and $D:\mathcal A\to E^*$
be a continuous derivation. Suppose that 
$\mathcal A=\underleftarrow{\lim}\mathcal A_n$ is an
 Arens-Michael decomposition of $\mathcal A$.
 Since $(E^*,\|\cdot\|)$ is a Banach space, by similar arguments to the
 proof of  Theorem \ref{T1}, there exists $n\in\mathbb{N}$
 and a continuous derivation on $\mathcal A_n$ defined by; 
 $$D_n:\mathcal A_n\to E^*,\;\;\; D_n(a_n)=D_n(a+\ker p_n)=D(a),\;\;\;\overline{I_n}.E=E.\overline{I_n}=0. $$
On the other hand since $\varphi_n$ is a continuous homomorphism we have; 
\begin{eqnarray*}
I_n=\varphi_n(I)=\varphi_n(I^2)&=&
\varphi_n( span\{ab:\;\;\;\;a,b\in I\})\\
&=& span\{\varphi_n(ab):\;\;\; a,b\in I\}\\
&=&span\{\varphi_n(a)\varphi_n(b):\;\;\; a,b\in I\}\\
&=&span\{a_nb_n:\;\;\; a_n,b_n\in I_n\}=I^2_n.
\end{eqnarray*}
Also ${I_n}.E^*=E^*.{I_n}=0$.
Therefore $I_n\subseteq \ker D_n$, but
$\ker D_n$ is a closed subspace of $\mathcal A_n$, so $\overline{I_n}\subseteq \ker D_n$.
Thus we can define a 
continuous derivation  $\widetilde{D_n}:\frac{{\mathcal A}_n}{\overline{I_n}}\to E^*$, by
$\widetilde{D_n}(a_n+\overline{I_n})=D_n(a_n)$.
On the other hand
$\frac{\mathcal A}{I}=\underleftarrow{\lim}\frac{\mathcal A_n}{\overline{I_n}}$
is an Arens-Michael decomposition of $\frac{\mathcal A}{I}$,  by \cite[Theorem 3.14]{F}.  
 Since $\frac{\mathcal A}{I}$ is amenable, so $\frac{\mathcal A_n}{\overline{I_n}}$ is an amenable Banach algebra
for each $n\in\mathbb{N}$  by
\cite[Theorem 9.5]{Pir}.
Therefore there exists $\eta_n\in E^*$ such that 
$\widetilde{D_n}(a_n+\overline{I_n})=(a_n +\overline{I_n}).\eta_n- \eta_n.(a_n +\overline{I_n})$ for each 
$a_n\in{\mathcal A}_n$.
So for each $a\in\mathcal A\setminus I$ we have
\begin{eqnarray*}
D(a)=D_n(a+\ker p_n)=D_n(a_n)&=&
\widetilde{D_n}(a_n+\overline{I_n})\\&=&(a_n +\overline{I_n}).\eta_n- \eta_n.(a_n +\overline{I_n})\\
&=& a_n.\eta_n-\eta_n.a_n\\
&=&a.\eta_n-\eta_n.a.
\end{eqnarray*}

Thus $\mathcal A$ is amenable modulo $I$.
\end{proof}

Some examples of amenable modulo an ideal of Banach algebras 
which are not amenable can be found in \cite{Rahimi  Amini}. Before we 
proceed to examples of amenability modulo an ideal for  Fr$\acute{e}$chet algebras
we give some necessary background material. 
We shall give the definitions and some basic properties of semigroup algebras.

Let $S$ be a semigroup and $s\in S$, and let $\delta_s$ denote the function on $S$
which is $1$ at $s$, and $0$ elsewhere. A  generic element of $\ell^1(S)$ is of the form
$$f=\sum_{s\in S}\alpha_s \delta_s, \;\;\;\sum_{s\in S}|\alpha_s|<\infty.$$
Now consider $f=\sum_{r\in S}\alpha_r\delta_r$ and $ g=\sum_{s\in S}\beta_s \delta_s\in \ell^1(S)$.
Set 
$$f\star g=\sum_{r\in S}\alpha_r\delta_r\star\sum_{s\in S}\beta_s \delta_s=\sum_{t\in S}(\sum_{rs=t}\alpha_r\beta_s)\delta_t,$$
where $\sum_{rs=t}\alpha_r\beta_s=0$ when there are no elements $r$ and $s$ in $S$
with $rs=t$. Then $(\ell^1(S), \star)$ is called the semigroup algebra  of $S$. Take 
$$\|f\|_1=\sum_{s\in S}|\alpha_s|.$$
Then  $(\ell^1(S), \star,\|\cdot\|_1)$  is a Banach algebra. If 
$\theta:S\to T$ is an ephimorphism of semigroups, then  by \cite{Dales}  there exists a contractive 
ephimorphism $\theta_{*}: \ell^1(S)\to  \ell^1(T)$ determined by 
$$\theta_{*}(\delta_s)=\delta_{\theta(s)},\;\;\;(s\in S).$$
Let $S$ be a semigroup. A relation $R$ on the set $S$ is called left[respectively, right]
compatible if $s,t,a\in S$ and $(s,t)\in R$ implies that $(as,at)\in R)$[ respectively, $ (sa,ta)\in R$] and it is called
compatible if $s, t, s^{'}, t^{'}\in S$ and $(s,t)\in R $ and $(s^{'},t^{'})\in R$ implies 
$(ss^{'},tt^{'})\in R$. A compatible equivalence relation is called congruence.
By \cite[Theorem 1.5.2]{howie} if $\rho$ is a congruence on the semigroup $S$, then
the quotient set $\frac{S}{\rho}$ is a semigroup with respect to the operation defined by
$$(a\rho)(b\rho)=(ab)\rho,\;\;\; (a,b\in S).$$
A congruence $\rho$ on $S$ is called a group congruence on $S$ if  $\frac{S}{\rho}$
is a group. We denote the least group congruence on $S$ by $\sigma$.
Also we denote the set of idempotent elements of $S$ by $E(S)$. A semigroup
$S$ is called an $E$-semigroup if $E(S)$ forms a subsemigroup of $S$ and $E$-inverse
if for all $s\in S$ there exists $t\in S$  such that $st\in E(S)$.
An inverse semigroup $S$ is called $E$-unitary if for each $s\in S$
and $e\in E(S)$,  $es\in E(S)$ implies $s\in E(S)$. According to \cite{Rahimi  Amini},
if $R$ is  a ring and $S$ is a semigroup, the semigroup ring $R[S]$ is the ring whose
elements are of the form $\sum_{s\in S}r_s s$, where $r_s\in R$ and all but 
infinitely many of the cofficient are zero. If $R=K$ is a field, then $K[S]$ is called a semigroup 
algebra. Now let $S$ be a  semigroup, $\rho$ a congruence on $S$ and $\pi: S\to\frac{S}{\rho}$
be the quotient map. Then one can extends $\pi$ to an algebra ephimorphim
$$\pi_{*}:K[S]\to K[\frac{S}{\rho}],$$
whose kernel $I_{\rho}$ is the ideal in $K[S]$, generated by the set 
$$\{s-t,\;\;\;s,t\in S\;\;\;with \;\;(s,t)\in \rho\}.$$
Hence $K[\frac{S}{\rho}]\cong\frac{K[S]}{I_{\rho}}$.

Now we provide some examples of Fr$\acute{e}$chet algebras which are amenable
modulo an ideal $I$, but are not amenable.
\begin{example}

Let $\mathcal S=(S_n,\theta_n^m)_{n,m\in\mathbb{N}}$ be an inverse sequence of semigroups such that 
the linking maps $\theta_{n}^{m}$ are onto.
Set $\mathcal {L}^1(\mathcal S)=\underleftarrow{\lim}(\ell^1(S_n), (\theta_n^m)_*)$. Clearly 
$\mathcal {L}^1(\mathcal S)$ is a Fr$\acute{e}$chet algebra.
\begin{enumerate}  

\item[(1)] If $S_n$ is an amenable $E$-unitary  inverse semigroup
with $E(S_n)$ infinite, then $\ell^1(S_n)$ is not amenable but is amenable  modulo $I_{\sigma_n}$, by \cite{Rahimi Amini} .
So  $\mathcal {L}^1(\mathcal S)$ is not amenable by \cite[Theorem 9.5]{Pir}, but
 by applying Theorem \ref{T1}, $\mathcal {L}^1(\mathcal S)$
is amenable modulo $I_{\sigma}=\underleftarrow{\lim}I_{\sigma_n}$.

 \item[(2)]
For each $n\in\mathbb{N}$, let $G_n$ be an  amenable group with identity $1_n$
and $T$ be an abelian semigroup with infinitely many idempotents which is not an inverse semigroup.
Also set $S_n=G_n\times T$.
Then by applying \cite[Example (iii)]{Rahimi Amini}, for each $n\in\mathbb{N}$,
 $\ell^1(S_n)$ is amenable modulo $I_{\sigma_n}$ but is not amenable Banach algebra.
  Therefore $\mathcal {L}^1(\mathcal S)$ 
is not amenable by \cite[Theorem 9.5]{Pir}, but by applying Theorem\ref{T1}, $\mathcal {L}^1(\mathcal S)$
is amenable modulo $I_{\sigma}=\underleftarrow{\lim}I_{\sigma_n}$.

All amenable  Fr$\acute{e}$chet algebras are amenable modulo $I$ for 
each closed ideal.
In the next example we give  a Fr$\acute{e}$chet algebra
which is amenable by  \cite[Corollary 9.8]{Pir}.

\item[(3)] All nuclear $\sigma-$ $C^*-$algebras are amenable and so amenable modulo an ideal $I$
for each closed ideal $I$. (see \cite[Corollary 9.8]{Pir} for more details)

In the next example we give  a Fr$\acute{e}$chet algebra which is not amenable 
and so is not amenable modulo $I=\{0\}$.

\item[(4)] Let $C^{\infty}([0,1])$ be the space of infinitely many differentiable  functions on $[0,1]$
with pointwise multiplication.Then $C^{\infty}([0,1])$ is a Fr$\acute{e}$chet algebra
with respect to the system of seminorms $p_n$ given by 
$$p_n(f)=2^{n-1}\sup\{|f^{(k)}(x)|:\;\;\;x\in[0,1], k=0,...,n-1\}.$$
 $C^{\infty}([0,1])$
is not weakly amenable   by
\cite[Theorem 1.3]{ARR1} and so 
is not amenable by \cite[Theorem 9.6]{Pir}. Therefore $C^{\infty}([0,1])$
is not amenable modulo $I=\{0\}$.\\
\end{enumerate}
\end{example}

\section{\bf Locally bounded approximate identity   modulo an ideal of a Fr$\acute{e}$chet algebra}
We recall from \cite{Rahimi1} the concept of bounded approximate identity modulo an 
ideal for a Banach algebra. A Banach algebra $\mathcal A$ has a bounded approximte identity 
modulo $I$ if there exists a bounded net $(u_{\alpha})_{\alpha}$ in
 $\mathcal A$ such that 
 $$\lim_{\alpha} u_{\alpha}a=\lim_{\alpha} au_{\alpha}=a, \;\;\;\;(a\in \mathcal A\setminus I).$$
Also in \cite{Pir} Pirkovskii asserts the concept of locally bounded approximate 
identity for a locally convex algebra. Similar to these definitions we 
define bounded and locally bounded approximate identities modulo an
ideal for a  Fr$\acute{e}$chet algebra.

\begin{definition}\label{D2}\rm
Let $(\mathcal A, p_{\lambda})$ be a locally 
convex algebra and 
 suppose that $I$ is a closed ideal
of  $\mathcal A$. A bounded net $(e_{\alpha})_{\alpha}$ is a bounded 
approximate identity modulo $I$  for $\mathcal A$, if 
$$\lim_{\alpha}p_{\lambda}(a e_{\alpha}-a)=\lim_{\alpha}p_{\lambda}
( e_{\alpha}a-a)=0,\;\;\;\;(a\in \mathcal A
\setminus I,\;\; \lambda\in \Lambda).$$ 
 Furthermore we say that $\mathcal A$ has a locally bounded approximate 
identity modulo $I$, if there exists a family $\{C_{\lambda} :\;\;\; \lambda\in\Lambda\}$
of positive real numbers such that for each finite set $F\subseteq \mathcal A\setminus I$,
 each  $\lambda\in\Lambda$, and each $\varepsilon>0$ 
there exists $b\in\mathcal A$  with $p_{\lambda}(b)\leq C_{\lambda}$
and $p_{\lambda}(a b-a)<\varepsilon$ and $p_{\lambda}( ba-a)<\varepsilon$, 
for all $a\in F$.
\end{definition}

We commence with the following proposition which gives us necessary and sufficient
conditions  for the exitence of a bounded approximate identity modulo an ideal for a 
 Fr$\acute{e}$chet algebra.
\begin{proposition}\label{P3}
Let $(\mathcal A, p_n)$ be a Fr$\acute{e}$chet algebra and $I$ be a closed ideal of
 $\mathcal A$. $\mathcal A$ has a bounded approximate identity modulo $I$ if and only if
there exists a bounded set $B\subseteq \mathcal A$ such that for each finite set 
$ F\subseteq \mathcal A\setminus I$, each $n\in\mathbb{N}$ and each $\varepsilon>0$ there exists 
$b\in  B$ such that 
$$p_n(ab-a)<\varepsilon \;\;\;\; and \;\;\;\;p_n(ba-a)<\varepsilon,$$
for each $a\in F$.
\end{proposition}

\begin{proof}
First suppose that $(e_{\alpha})_{\alpha}\subseteq  {\mathcal A}$
is a bounded approximate identity modulo $I$ for $\mathcal A$.
So for   each  $n\in \mathbb{N}$ and each $\varepsilon>0$ 
there exists $\alpha_0$ such that for each $\alpha\geq \alpha_0$ we have 
$${p}_n(a e_{\alpha}-a)<\varepsilon \;\;\; and\;\;\;\;  p_n(e_{\alpha}a-a)<\varepsilon,\;\;\;\;(a\in \mathcal A\setminus  I).$$
Now  set $B=(e_{\alpha})_{\alpha}\subseteq  {\mathcal A}$
and let $n\in \mathbb{N}$, $\varepsilon>0$  and $F\subseteq \mathcal A\setminus I$
be a finite set. Therefore there exists $b=m_{\alpha}$  which ${\alpha\geq \alpha_0}$
such that 
$$p_n(ab-a)<\varepsilon\;\;\;\;and \;\;\;
{p}_n(ba-{a})<\varepsilon,\;\;\;(a\in F). $$
Conversely, Suppose that there exists a bounded set $B\subseteq {\mathcal A}$
such that   for each finite set $F\subseteq \mathcal A\setminus I$,
 each  $n\in \mathbb{N}$ and each $\varepsilon>0$ 
there exists $b\in B$  with
$$p_n(ab-a)<\varepsilon\;\;\; and \;\;\; 
{p}_n(ba-a)<\varepsilon,\;\;\;(a\in F).$$
Now take 
$$S=\{(n,\varepsilon, F):\;\;n\in\mathbb{N},\;\;\varepsilon>0,\;\; F\subseteq \mathcal A\setminus I \;\;  is \;\;a\;\; finite\;\; set  \}.$$
So $S$ is a directed set as follows:
$$(n_1,\varepsilon_1, F_1)\leq (n_2,\varepsilon_2, F_2)\;\;\Leftrightarrow\;\;n_1\leq n_2,\;\;\varepsilon_2\leq\varepsilon_1, \;\;F_1\subseteq  F_2.$$
For each $\alpha=(n,\varepsilon, F)$, there exists $b=e_{\alpha}\in B$ and so 
we have a bounded net $(e_{\alpha})\subseteq {\mathcal A}$.
Furthermore let $k\in\mathbb{N}$ and $\varepsilon>0$ and the finite set $F\subseteq \mathcal A\setminus I$
be arbitrary. Therefore for each $(n,\delta, C)\geq(k,\varepsilon, F)$
we have
$${p}_{k}(a e_{\alpha}-{a})\leq{p}_n(a e_{\alpha}- {a})<\delta<\varepsilon,$$
$${p}_{k}(e_{\alpha}a -{a})\leq{p}_n(e_{\alpha}a- {a})<\delta<\varepsilon,\;\;\;\;(a\in \mathcal A\setminus  I). $$

So $\mathcal A$ has a bounded approximate identity modulo $I$.
\end{proof}
\begin{remark}\label{rem1}
In view of Proposition \ref{P3}, if $\mathcal A$ is normable, then the notions of bounded and  locally bounded approximate
identity modulo $I$ are equivalent.
\end{remark}
By \cite[Theorem 4]{Rahimi1} it is known that if a Banach algebra $\mathcal{A}$ is amenable modulo an 
ideal $I$, it has a bounded approximate identity modulo $I$. The following results are interesting
in theirs own right. In fact we use them to extend \cite[Theorem 4]{Rahimi1}
for Fr$\acute{e}$chet algebrs.
\begin{proposition}\label{P2}
Let  $\varphi :\mathcal A\to \mathcal B$ be a continuous homomorphism
of  Fr$\acute{e}$chet algebras with dense range and let 
$I$ be a closed ideal of $\mathcal A$ and $J$ be a closed ideal of $\mathcal B$
such that $\varphi(I)\subseteq J$.  Suppose that
$\mathcal{A}$ has a  locally bounded approximate identity modulo $I$.
 Then $\mathcal B$ has a locally bounded approximate identity  modulo $J$.
\end{proposition}

\begin{proof}
Let $\{p_n:\; n\in\mathbb{N}\}$ be a family of fundamental system of seminorms
generating the topology of $\mathcal A$ and  $ \{q_m:\; m\in\mathbb{N}\}$
generating the topology of $\mathcal B$. Continuity of $\varphi$
implies for each $m\in\mathbb{N}$ the existence of $k_m>0$ and 
$n_0\in\mathbb{N} $  such that
$$q_m(\varphi(a))\leq k_m p_{n_0}(a),\;\;\;(a\in\mathcal A).$$
Since $\mathcal A$ has a locally bounded approximate identity modulo $I$
one concludes from Definition \ref{D2}, there exists a family $\{C_n:\;\;n\in\mathbb{N}\}$
of positive real numbers, such that for each finite set $F\subseteq \mathcal A\setminus I$,
each $n\in\mathbb{N}$ and each $\varepsilon>0$ there exists $a^{'}\in\mathcal A$
such that
 $$p_n(a^{'})\leq C_n\;\;\;and\;\;\; p_n(a-a a^{'})<\frac{\varepsilon}{3k_m}\;\;\; and \;\;\;p_n(a-a^{'}a)<\frac{\varepsilon}{3k_m},$$
for each $a\in F$.
Without loss of generality we may assume that $k_m $ and $ C_{n_0}>1$ for each $m\in\mathbb{N}$.
Now consider the family $\{k_mC_{n_0}:\;\;m\in\mathbb{N}\}$ of positive real numbers.
Given a finite set $F^{'}\subseteq \mathcal B\setminus J$,  $\varepsilon>0$ and $m\in\mathbb{N}$, 
find a finite set  $F\subseteq \mathcal A\setminus I$ such that, $q_m(\varphi(a)-b))<\frac{\varepsilon}{3k_mC_{n_0}}$,
for each $b\in F^{'}$ and $a\in F$. Now set $b^{'}=\varphi(a^{'})$. From our assumption
it follows that 
$$q_m(b^{'})=q_m(\varphi(a^{'}))\leq k_m p_{n_0}(a^{'})\leq k_m C_{n_0}$$
and 
\begin{eqnarray*}
q_m(b-bb^{'})&\leq&
q_m(\varphi(a-aa^{'}))+q_m(b-\varphi(a))+q_m((\varphi(a)-b)b^{'})\\
&\leq& k_m p_{n_0}(a-a a^{'})+\frac{\varepsilon}{3}+ q_m(\varphi(a)-b)q_m(b^{'})\\
&\leq& k_m\frac{\varepsilon}{3k_m}+\frac{\varepsilon}{3}+ \frac{\varepsilon}{3k_mC_{n_0}}k_m C_{n_0}\\
&=&\varepsilon,
\end{eqnarray*}
and similarly $q_m(b-b^{'}b)<\varepsilon$, for each $b\in F^{'}$.
\end{proof}

By using   Proposition \ref{P2} and \cite[Remark 6.2]{Pir}  the following is immediate.
\begin{corollary}
Let $(\mathcal A, p_n)$ be a Fr$\acute{e}$chet algebra and $I$ be a closed ideal of
 $\mathcal A$.  If $\mathcal A$ has a locally bounded approximate identity modulo $I$, then 
$\frac{\mathcal A}{I}$ has a locally bounded approximate identity.
\end{corollary}
An immediate cosequence of Proposition \ref{P2} and Remark \ref{rem1} is the following result.
\begin{corollary}\label{Cor2}
Let  $\mathcal A$ and   $\mathcal B$ be two Banach algebras and 
$\varphi :\mathcal A\to \mathcal B$ be a continuous homomorphism
 with dense range. Suppose that  
$I$ is a closed ideal of $\mathcal A$ and $J$ is a closed ideal of $\mathcal B$
such that $\varphi(I)\subseteq J$.  If $\mathcal A$
has a   bounded approximate identity modulo $I$,
then $\mathcal B$ has a bounded approximate identity  modulo $J$.
\end{corollary}
It is not hard to see the following lemma holds.
\begin{lemma}\label{Lem1}
Let $D$ be a dense subspace of normed algebra $\mathcal A$. If $\mathcal A$
has a bounded approximate identity $(e_{\alpha})_{\alpha}$ modulo $I$, then $\mathcal A$
has a bounded approximate identity $(f_{\mu})_{\mu}$ modulo $I$ such that
$f_{\mu}\in D$ for each $\mu$.
\end{lemma}
In the sequel we shall use the notation ${I_n}={\varphi_n(I)}$, where 
$\varphi_n:\mathcal A\to\mathcal A_n$ is the canonical map.

\begin{proposition}\label{P4}
Let $(\mathcal A,p_n)$ be a  Fr$\acute{e}$chet algebra and $I$ be a closed ideal of
 $\mathcal A$. Then the following statements are equivalent;
\begin{enumerate}
\item[(i)]   $\mathcal A$  has a locally bounded approximate identity modulo $I$.
 \item[(ii)]
For each Banach algebra $\mathcal B$ such that there exists a continuous homomorphism 
$\varphi:\mathcal A\to\mathcal B$ with dense range, $\mathcal B$ has a  bounded 
approximate identity modulo $J$, where $J$ is a closed ideal of $\mathcal B$
 such that $\varphi(I)\subseteq J$.
\end{enumerate}
\end{proposition}
\begin{proof}
$(i)\Rightarrow (ii)$. This follows by using Proposition \ref{P2} and 
viewing Remark \ref{rem1}.\\
$(ii)\Rightarrow (i)$. Let $n\in\mathbb{N}$ be arbitrary and 
let $\mathcal A=\underleftarrow{\lim}\mathcal A_n$ be an
 Arens-Michael decomposition of $\mathcal A$
and $I=\underleftarrow{\lim}\overline{I_n}$ be 
an Arens-Michael decomposition of $I$. Since $\varphi_n:\mathcal A\to\mathcal A_n$
is a continuous homomorphism with dense range, by assumption the Banach algebra $\mathcal A_n$
has a bounded approximate identity modulo $\overline{I_n}$. Hence by Lemma 
\ref{Lem1}, the same is true for the dense subalgebra $\varphi_n(\mathcal A)\subseteq \mathcal A_n$.
Therefore  Proposition \ref{P3}, yields  the exitence of  a bounded set $B\subseteq \mathcal A$
such that for each finite set $F^{'}\subseteq \varphi_n(\mathcal A)\setminus \overline{I_n}$
and each $\varepsilon>0$ there exists $b\in B$ such that 
$$\|b^{'}\varphi_n(b)- b^{'}\|_n<\varepsilon \;\;\;and\;\;\; \|\varphi_n(b)b^{'}- b^{'}\|_n<\varepsilon,$$
for each $b^{'}\in F^{'}$.
 Note  that if $F=\{a_1, a_2, ...,a_m\}\subseteq \mathcal A\setminus I$ is a
finite set, then there exists $n_0\in\mathbb{N}$ such that for each $n\geq n_0$, 
$a_n^i\notin \overline{I_n}$, where $a_i=(a_n^i)$ and $1\leq i\leq m$. 
In fact if $a_i\notin I$, then there exists $n_i\in\mathbb{N}$,
such that $a_{n_i}^i\notin \overline{I_{n_i}}$. So for each $n\geq n_i$, $a_n^i\notin \overline{I_n}$.
Now  put $n_0=\max\{n_i:\;\;1\leq i\leq m\}$.
Therefore for each $n\geq n_0$,  $a_{n}^i\notin \overline{I_{n}}$.
In the sequel take  a finite set $F\subseteq \mathcal A\setminus I$, $k\in\mathbb{N}$ and
 $\varepsilon>0$. Then for each $n\geq \max\{k, n_0\}$, we can find a finite 
set $F^{'}=\varphi_n(F)\subseteq \varphi_n(\mathcal A)\setminus \overline{I_n}$,
such that
\begin{eqnarray*}
p_k(ab-a)&\leq & p_n(ab-a)\\
&=&
\|\varphi_n(ab-a)\|_n\\
&=& \|\varphi_n(a)\varphi_n(b)-\varphi_n(a)\|_n\\
&<&\varepsilon
\end{eqnarray*}
and 
\begin{eqnarray*}
p_k(ba-a)&\leq & p_n(ba-a)\\
&=&
\|\varphi_n(ba-a)\|_n\\
&=& \|\varphi_n(b)\varphi_n(a)-\varphi_n(a)\|_n\\
&<&\varepsilon,
\end{eqnarray*}
for each $a\in F$.  Since $b\in B$ and $B$ is a bounded set, for 
each $k\in\mathbb{N}$ there exists $C_k>0$ such that 
$p_k(b)\leq C_k$. 
By Definition \ref{D2}, this completes the proof.
\end{proof}

\begin{corollary}\label{Cor4}
Let $(\mathcal A,p_n)$ be a  Fr$\acute{e}$chet algebra and $I$ be a closed ideal of
 $\mathcal A$. Suppose that $\mathcal A=\underleftarrow{\lim}\mathcal A_n$
be an Arens-Michael decomposition of $\mathcal A$ and $I=\underleftarrow{\lim}\overline{I_n}$ is an 
Arens-Michael decomposition of $I$. Then $\mathcal A$ 
has a locally bounded approximate identity modulo $I$ if and only if
each  $\mathcal A_n$ has a bounded approximate identity modulo $\overline{I_n}$. 
\end{corollary}
\begin{proof}
Since $\varphi_n:\mathcal A\to \mathcal A_n$ is a continuous homomorphism
with dense range, Proposition \ref{P4} implies that if $\mathcal A$ has a locally
 bounded approximate identity modulo $I$, then $\mathcal A_n$ has a
 bounded approximate identity modulo $\overline{I_n}$.
\\Conversely, suppose that each $\mathcal A_n$ has a bounded approximate identity
modulo $\overline{I_n}$ and $\mathcal B$ be a Banach algebra such that there exists 
a continuous homomorphism $\Phi:\mathcal A\to \mathcal B$ with dense range. Arguing 
as in the   proof of Theorem \ref{T1}, we deduce that there exists a 
continuous homomorphism with dense range $\Phi_n:\mathcal A_n\to \mathcal B$,
for some $n\in\mathbb{N}$. On the other hand if $J$ is a closed ideal of $\mathcal B$
such that $\Phi(I)\subseteq J$, we have $\Phi_n(I_n)=\Phi(I)\subseteq J$.
Therefore by Corollary \ref{Cor2}, $\mathcal B$ has a bounded approximate identity
modulo $J$ such that $\Phi(I)\subseteq J$. Thus by using Proposition \ref{P4},
$\mathcal A$ has a locally bounded approximate identity modulo $I$.
\end{proof}
 Now we are in position to prove \cite[Theorem 4]{Rahimi1} for  Fr$\acute{e}$chet algebras.
\begin{corollary}
Let $\mathcal A$ be a  Fr$\acute{e}$chet algebra and $I$ be a closed ideal of 
$\mathcal A$. If $\mathcal A$ is amenable modulo $I$, then $\mathcal A$ 
has a locally bounded approximate identity modulo $I$.
\end{corollary}
\begin{proof}
Suppose that $\mathcal A=\underleftarrow{\lim}\mathcal A_n$
be an Arens-Michael decomposition of $\mathcal A$ and $I=\underleftarrow{\lim}\overline{I_n}$ be an 
Arens-Michael decomposition of $I$. Since $\mathcal A$ is amenable modulo $I$, 
each $\mathcal A_n$ is amenamble modulo $\overline{I_n}$, by Theorem \ref{T1} .
So   $\mathcal A_n$ has a bounded approximate
identity  modulo $\overline{I_n}$ for each $n\in\mathbb{N}$, by \cite[Theorem 4]{Rahimi1}. In view of Corollary \ref{Cor4}, this completes the proof.
\end{proof}

{\bf Acknowledgment.} This work would not have been possible without the financial support of
the Iran national science foundation, which have been supportive our
goals. We would like to express our special gratitude and thanks to it for giving us such attention and time.
The authors's  thanks and appreciations also go to the university of Isfahan and the center
 of excellence for mathematics at the university of Isfahan
 in developing the project.

\vspace{9mm}

{\footnotesize \noindent
 S. Rahnama\\
  Department of Mathematics,
   University of Isfahan,
    Isfahan, Iran\\
     rahnama$_{-}600$@yahoo.com\\

\noindent
 A. Rejali\\
  Department of Mathematics,
   University of Isfahan,
    Isfahan, Iran\\
    rejali@sci.ui.ac.ir\\

\end{document}